\documentclass[12pt , a4j]{amsart}
\usepackage{ifpdf}
\usepackage{amsmath, amsthm , amssymb}
\usepackage[abbrev]{amsrefs}
\usepackage[dvipdfm]{graphicx}
\usepackage{color}
\usepackage{comment}
\usepackage{float}
\usepackage{amscd}
\usepackage{stmaryrd}
\usepackage{subcaption}
\usepackage{latexsym}
\usepackage{fancybox}
\usepackage{amsfonts}
\usepackage[top=17truemm,bottom=18truemm,left=29truemm,right=29truemm]{geometry}

\theoremstyle{definition}
\newtheorem{thm}{Theorem}
\newtheorem{prop}[thm]{Proposition}
\newtheorem{lem}[thm]{Lemma}

\newtheorem{dfn}[thm]{Definition}

\newtheorem{rem}[thm]{Remark}



\newcommand{\R}{\mathbf{R}}

\newcommand{\Z}{\mathbb{Z}}

\def\Mod{\operatorname{Mod}}

\begin{document}


\title[Right-left equivalent maps of simplified $(2, 0)$-trisections]{Right-left equivalent maps of simplified $(2, 0)$-trisections with different configurations of vanishing cycles}
\author{Nobutaka Asano}
\address{Mathematical Institute, Tohoku University, Sendai, 980-8578, Japan}
\email{nobutaka.asano.r4@dc.tohoku.ac.jp}

\begin{abstract}
Trisection maps are certain stable maps from closed $4$--manifolds to $\R^2$.
A simpler but reasonable class of trisection maps was introduced by Baykur and Saeki, called a simplified $(g, k)$-trisection.
We focus on the right-left equivalence classes of simplified $(2, 0)$-trisections.
Simplified trisections are determined by their simplified trisection diagrams, which are diagrams on a genus-$2$ surface consisting of simple closed curves of vanishing cycles with labels.
The aim of this paper is to study how the replacement of reference paths changes simplified trisection diagrams up to upper-triangular handle-slides.
We show that for a simplified trisection $f$ satisfying a certain condition, there exists at least two simplified $(2, 0)$-trisections $f'$ and $f''$ such that $f, f'$ and $f''$ are right-left equivalent to each other but their simplified trisection diagrams are not related by automorphisms of a genus-$2$ surface and upper-triangular handle-slides.
\end{abstract}
\maketitle

\section{Introduction}
A trisection map is a certain stable map from a connected closed orientable smooth $4$--manifold to the $2$--dimensional plane.
The singular value set of a stable map of a trisection is the union of immersed circles with cusps as in Figure~\ref{tri}, where the singular value set in the white boxes consists of immersed curves with only normal double points, without cusps and not to be in contact with radial directions. 
Such a stable map is called a trisection map.
Trisection maps are introduced by Gay and Kirby \cite{Kirby}.
Trisecting the target of a trisection map, we obtain a handle decomposition of the source $4$--manifold, which is a $4$-dimensional analog of a Heegaard splitting of a $3$-manifold.
On the other hand, a trisection map can be seen as a specific description of a $4$--manifold in terms of stable maps.
As a study of stable maps, it is natural to ask when two trisection maps are equivalent as stable maps.
The fiber $\Sigma$ over the center point $p_0$ in Figure~\ref{tri} is a closed orientable surface of genus $g$, and this has the highest genus among all regular fibers of the trisection map. 
Set $k$ to be the number of arcs in Figure~\ref{tri} connecting neighboring white boxes without cusps. 
In this case, the trisection is called a $(g, k)$-trisection.
The vanishing cycles of indefinite folds of a trisection map can be represented by simple closed curves on the central fiber $\Sigma$. 
The vanishing cycles of a trisection determine the source $4$-manifold up to diffeomorphism. 
The surface $\Sigma$ with these simple closed curves is called a trisection diagram.

\begin{figure}[htbp]
\begin{center}
\includegraphics[clip, width=5cm, bb=152 424 441 713]{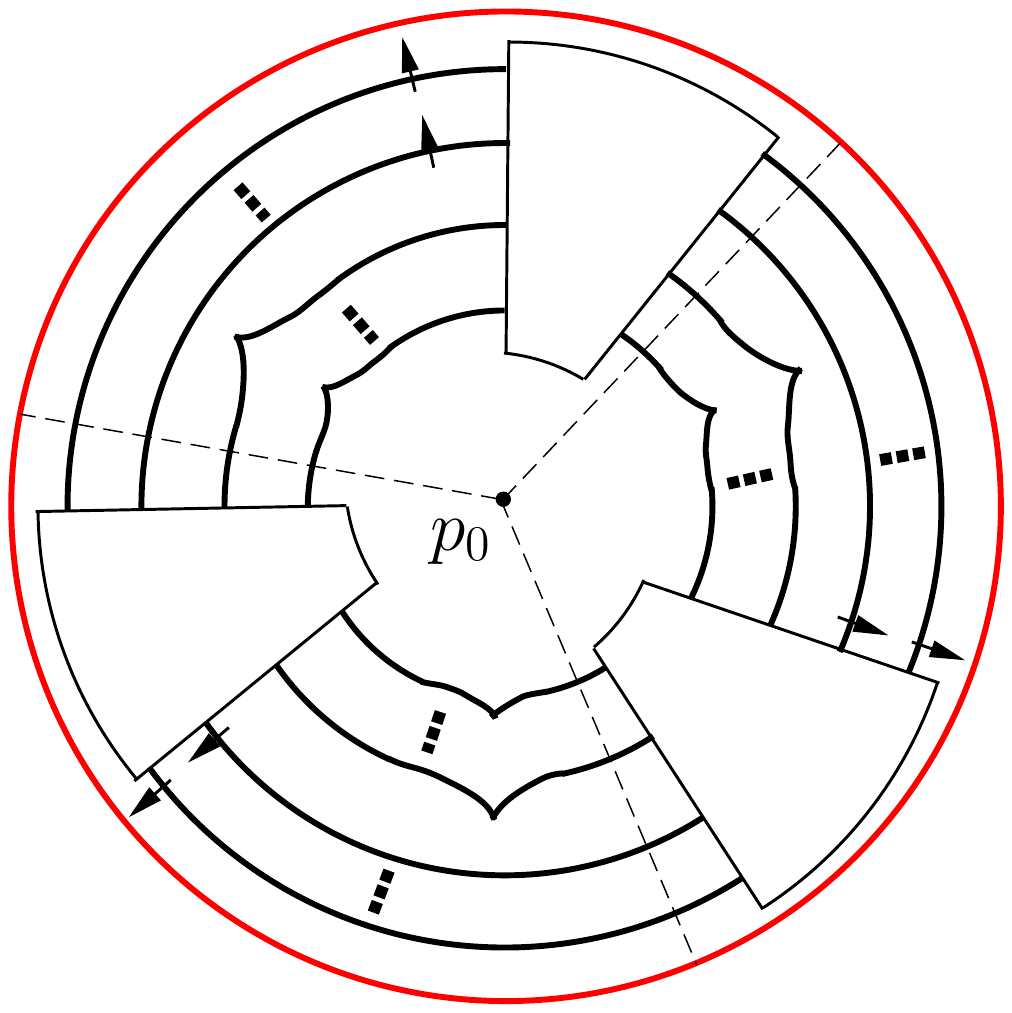}
\caption{A singular value set of a trisection map.}
\label{tri}
\end{center}
\end{figure}

In \cite{Kirby}, it is proved that two trisection diagrams of the same $4$--manifold are related by automorphisms of the central fiber, handle-slides and stabilizations.
The classification of trisection diagrams on a genus-$1$ surface up to handle-slides is easy, and that on a genus-$2$ surface up to handle-slides had been done by Meier and Zupan in \cite{MZ}. 
The genus-$g$ case for $g \geq 3$ is still difficult.

Simplified trisections are specific trisection maps whose singular value sets have no normal double points in the white boxes in Figure~\ref{tri}, which is introduced by Baykur and Saeki~\cite{BS2}.
They proved the existence of simplified trisections from simplified broken Lefschetz fibrations on closed $4$--manifolds by using certain singularity theoretical technique. 
A difference between trisection maps and simplified ones can be seen in the usage of handle-slides (c.f. \cite[Remark 7.2]{BS}).

The classification of $4$--manifolds admitting simplified genus-$2$ trisections follows from the Meier-Zupan's classification of genus-$2$ trisections \cite{MZ} and Baykur and Saeki's construction of simplified trisections \cite{BS}. 
Hayano gave a short proof of this classification, which can be attributed to a problem in linear algebra by deforming simplified $(2, 0)$-trisections \cite{hayano}.
The auther focused on 3-dimensional manifolds, called {\it vertical $3$-manifolds}, obtained as the preimages of properly embedded arcs in the target of a simplified $(2, 0)$-trisection. 
He classified the vertical $3$-manifolds of simplified genus-$2$ trisections and studied if they determine the source $4$-manifold. 

To see configurations of vanishing cycles on the central fiber from a viewpoint of stable maps, we need to fix a reference path in the target space.
In \cite{hayano}, a specific reference path is chosen, which we call the {\it standard reference path} in this paper.
For a simplified $(g, k)$-trisection, there are $3(g - k)$ simple closed curves of vanishing cycles on the central fiber, each of which corresponds to a connected component of the image of indefinite folds adjacent to the cusps, and $k$ simple closed curves of vanishing cycles, each of which corresponds to a round circle in the image of indefinite folds. 
The $3(g - k )$ simple closed curves are labeled with $\{ a_1, a_2, \ldots, a_{g-k} \}, \{ b_1, b_2, \ldots, b_{g-k} \}, \{ c_1, c_2, \ldots, c_{g-k} \}$ and the $k$ simple closed curves are labeled with $\{ a_{g-k+1}, a_{g-k+2}, \ldots a_{g}\}$.
In this paper, we call the genus-$g$ surface with these simple closed curves with labels a {\it simplified trisection diagram}.
It is observed in \cite{BS, hayano} that the isotopy class of each of the vanishing cycles $a_i, b_i$ and $c_i$ is well-defined up to handle-slides along $a_1, a_2, \ldots, a_{i - 1}$.
These handle-slides are called {\it upper-triangular handle-slides}.

In this paper, we focus on the classification of right-left equivalence classes of simplified $(2, 0)$-trisections.
We can see that if simplified trisection diagrams of simplified $(2, 0)$-trisections are not related by automorphisms of the central fiber $\Sigma$, upper-triangular handle-slides over $a_1$, and replacing reference paths, then they are not right-left equivalent (Lemma~\ref{differentRightLeftEq}).
A handle-slide is a certain systematic operation for vanishing cycles on a simplified trisection diagram.
On the other hand, the replacement of reference paths changes the configurations of vanishing cycles drastically.
The aim of this paper is to study how the replacement of reference paths changes simplified trisection diagrams up to upper-triangular handle-slides.

Let $\Sigma_{a_1}$ denote the torus obtained from the central fiber $\Sigma$ of a simplified $(2, 0)$-trisection by surgering it along the simple closed curve $a_1$ on $\Sigma$.
Here ``surgery'' means ``cut $\Sigma$ along $a_1$ and fill the appearing two boundary components by disks''.
Let $\mu_1$ be the monodromy along a circle between two three-cusped circles of the singular value set, which is an automorphism of $\Sigma_{a_1}$.
If $\mu_1$ is the identity map, then the trisection is in some sense ``trivial'', see \cite{asano}.
On the other hand, we can find non-trivial replacements of reference paths when $\mu_1$ is not the identity map.
Let $a'_2, b'_2$ and $c'_2$ be the vanishing cycles on $\Sigma_{a_1}$ obtained from $a_2, b_2$ and $c_2$ by the surgery from $\Sigma$ to $\Sigma_{a_1}$, respectively.

The main theorem is the following:


\begin{thm}
Let $f : X \to \R^2$ be a simplified $(2, 0)$-trisection. 
Suppose that the following holds:
	\begin{itemize}
		\item The monodromy $\mu_1$ is not the identity map. 
		\item $b'_2$ and $c'_2$ are not parallel. 
		\item $a'_2$ and $\mu_1^{-1}(c'_2)$ are not parallel.
	\end{itemize}
Then there exists at least two simplfied $(2, 0)$-trisections $f'$ and $f''$ such that $f, f'$ and $f''$ are right-left equivalent to each other but their simplified trisection diagrams are not related by automorphisms of $\Sigma$ and upper-triangular handle-slides over $a_1$.
\end{thm}

It will be proved that any reference path can be obtained from the standard one by applying  two kinds of operations called a $\Delta_1$-move and a $\Delta_2$-move successively up to isotopy on $\R^2$ fixing the singular value set.
The set of reference paths has a structure of a group generated by these operations.
$\Delta_2$-moves induce an action of $\Z / 3\Z$ to the set of configurations of vanishing cycles modulo automorphisms of $\Sigma$ and upper-triangular handle-slides.
If the monodromy $\mu_1$ is not the identity map, then we may prove that the action induced by $\Delta_2$-moves is non-trivial, especially.
This is the strategy of our proof of the main theorem.

The author would like to thank Masaharu Ishikawa for many discussions and encouragement. He would also like to thank Kenta Hayano for giving me some useful suggestions early in this study. This work is supported in part by the WISE Program for AI Electronics, Tohoku University.
\section{Preliminaries}
Let $f$ and $g$ be smooth maps from an $m$-dimensional manifold $M$ to an $n$-dimensional manifold $N$. The maps $f$ and $g$ are said to be right-left equivalent if there exist self-diffeomorphisms $\phi$ on $M$ and  $\psi$ on $N$ such that the diagram
$$
\begin{CD}
M @> \phi >> M\\
@VfVV @VVgV \\
N @ > \psi >> N
\end{CD}
$$
is commutative.
A smooth map $f$ is called a stable map if there exists a neighborhood $U$ of $f$ in the space of smooth maps from $M$ to $N$, with the Whitney $C^{\infty}$-topology, such that any map $g \in U$ is right-left equivalent to $f$.
 
Let $X$ be a closed orientable smooth $4$--manifold and $f : X \to \R^2$ be a stable map.
Singularities of $f$ are classified into four types: definite folds, indefinite folds, definite cusps and indefinite cusps.
The image of indefinite folds is an immersed curve on $\mathbf{R}^2.$ 
For an indefinite fold $x$, choose a short arc on $\R^2$ that intersects the image of indefinite folds transversely at $f(x)$. 
The fiber changes along this short arc as shown on the left in Figure~\ref{singularities}. 
The image of definite folds is also an immersed curve on $\mathbf{R}^2$ and the fiber changes along a transverse short arc as shown on the middle. 
Cusps appear at the endpoints of folds. The singular value set near the image of each cusp is a cusped curve as shown on the right.

\begin{figure}[htbp]
\begin{center}
\includegraphics[clip, width=15cm, bb=129 591 500 713]{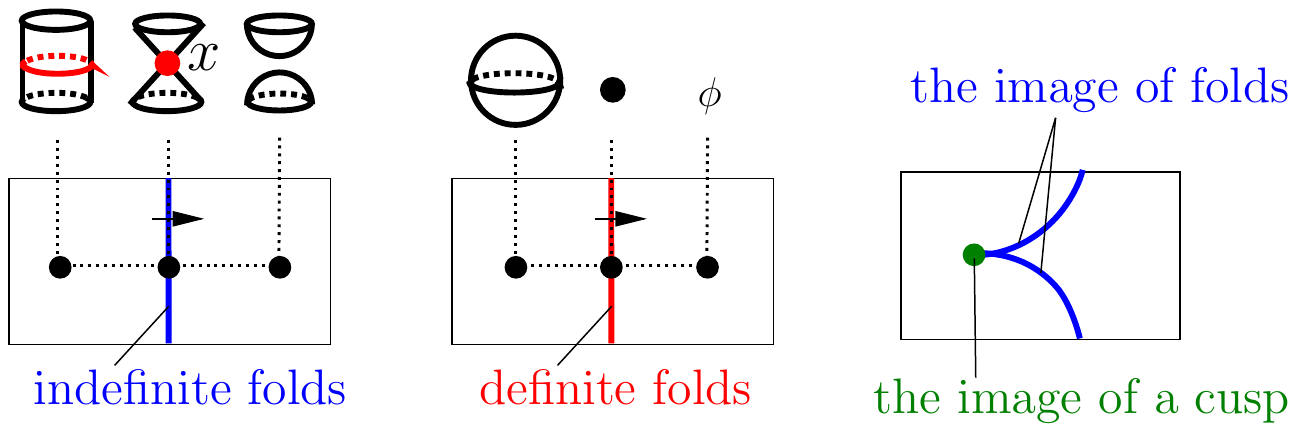}
\end{center}
\caption{Deformation of fibers near singularities.}
\label{singularities}
\end{figure}

\begin{dfn} 
A stable map $f : X \to \mathbf{R}^2$ is called a simplified $(g, k)$-trisection if the following conditions hold:
\begin{itemize}
\item The singular value set of definite folds is a circle, bounding a disk $D^2$.
\item The closure of the singular value set of indefinite folds consists of $g$ concentric circles on $D^2$ each of whose inner $g-k$ circles has three outward cusps.
\item The preimage of the point at the center is a closed orientable surface of genus $g$.
\end{itemize}
See Figure~\ref{sim_tri_2}.
\end{dfn}

\begin{figure}[htbp]
\begin{center}
\includegraphics[clip, width=5.0cm, bb=274 499 437 662]{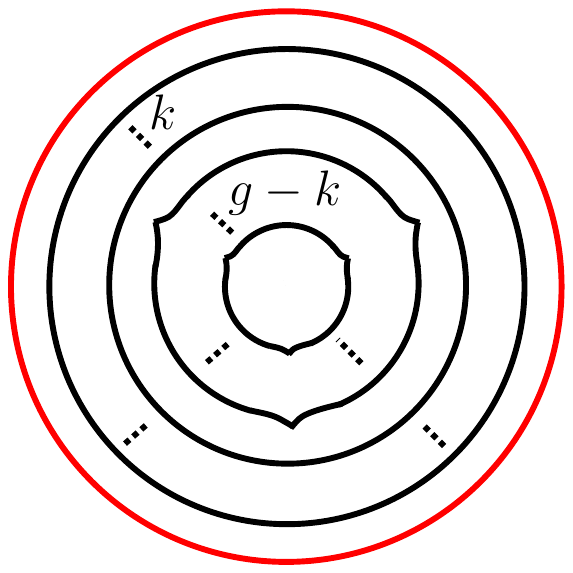}
\end{center}
\caption{A simplified $(g, k)$-trisection.}
\label{sim_tri_2}
\end{figure}

\begin{dfn}
A tree in $\R^2$ shown by dotted segments in Figure~\ref{refPath} is called the {\it standard reference path}. 
\end{dfn}

\begin{figure}[htbp]
\begin{center}
\includegraphics[clip, width=8.0cm, bb=129 377 465 712]{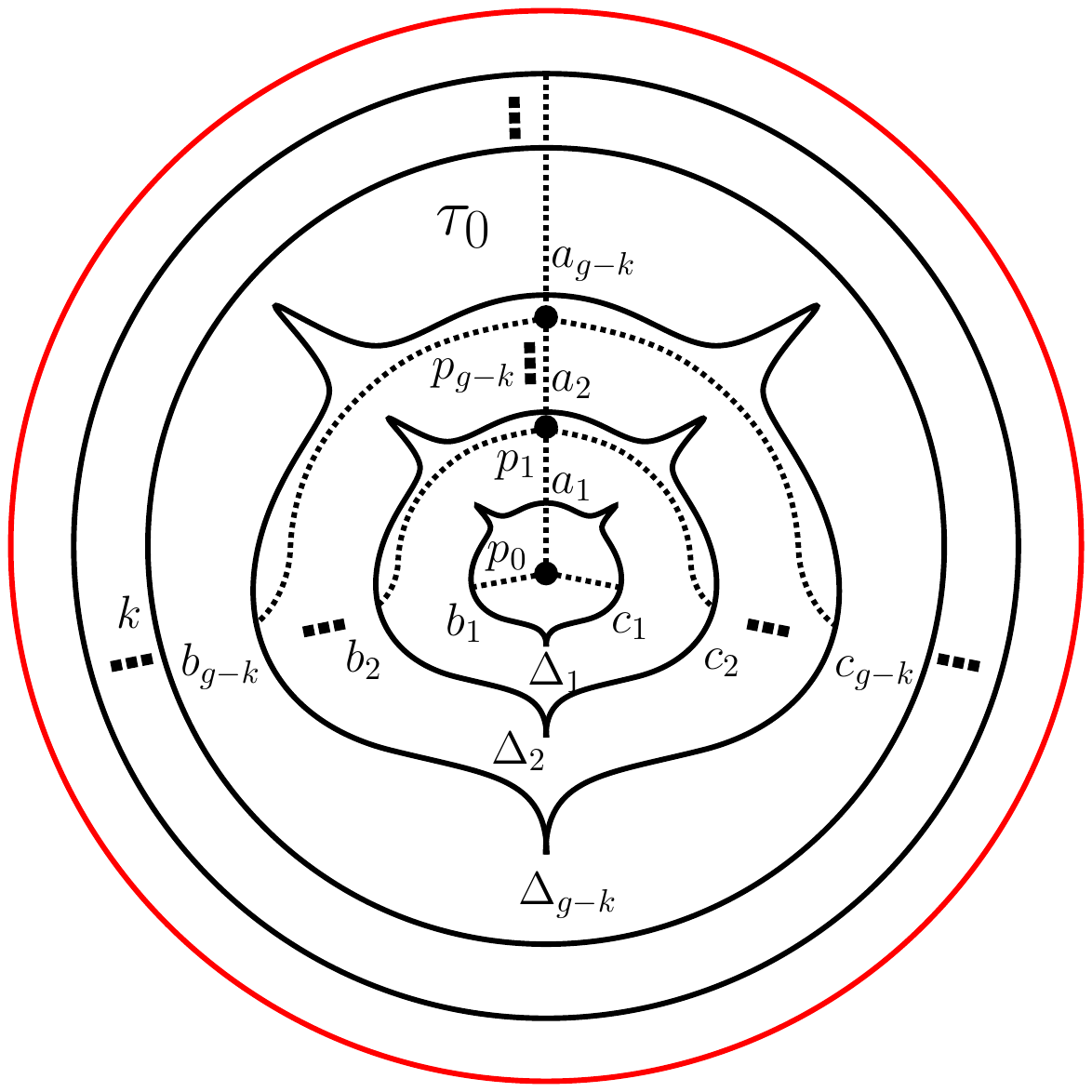}
\end{center}
\caption{The standard refernce path of a simplified $(g, k)$-trisection.}
\label{refPath}
\end{figure}


Let $f:X\to \R^2$ be a simplified $(g,k)$--trisection. 
The image of indefinite folds has $g - k$ three-cusped circles.
We denote them by $\Delta_1, \Delta_2, \ldots, \Delta_{g-k}$ from the innermost one in order.
Let $p_0$ be the center of $D^2$.
Let $\tau_0$ be the standard reference path and $p_i$ be the four-valent vertex of $\tau_0$ between $\Delta_i$ and $\Delta_{i+1}$ for $i = 1, 2, \ldots, g - k - 1$.
We assign the labels $a_i, b_i$ and $c_i$ to the intersection points of $\tau_0$ and $\Delta_i$ for $i = 1, 2, \ldots, g - k$ as follows:
	\begin{itemize}
		\item The intersection point of $\Delta_1$ and the edge of $\tau_0$ connecting $p_0$ and $p_1$ is labeled with $a_1$.
		\item The other two intersection points of $\Delta_1$ and $\tau_0$ are labeled with $b_1$ and $c_1$ such that $a_1, b_1, c_1$ are aligned counter-clockwise along $\Delta_1$. 
		\item For each $i = 2, \ldots, g -k$, the three intersection points of $\Delta_i$ and $\tau_0$ are labeled with $a_i, b_i, c_i$ such that  $a_i, b_i, c_i$ are aligned counter-clockwise along $\Delta_i$ 
		and the edge of $\tau_0$ connecting $p_{i-2}$ and $p_{i-1}$ is between the edges of $\tau_0$ adjacent to the points labeled with $b_i$ and $c_i$.
	\end{itemize}

We denote the fiber on the center $p_0$ by $\Sigma$. 
If one moves from $p_0$ to $a_1$ on $\tau_0$, then a simple closed curve on $\Sigma$ shrinks to a point. 
This simple closed curve is called a vanishing cycle associated with $\tau_0$.
For simplicity, we denote this vanishing cycle on $\Sigma$ again by $a_1$.
We may define the vanishing cycles for $b_1$ and $c_1$ similarly and we denote them again by $b_1$ and $c_1$, respectively.
Let $\Sigma_{a_1}$ denote the closed surface obtained from $\Sigma$ by surgering it along the simple closed curve $a_1$ on $\Sigma$.
Here ``surgery'' means ``cut $\Sigma$ along $a_1$ and fill the appearing two boundary components by disks''.
The surface $\Sigma_{a_1}$ is the fiber over the point $p_1$ in Figure \ref{refPath}.

If one moves from $p_0$ to $a_2$ on $\tau_0$, the surface $\Sigma$ first changes to $\Sigma_{a_1}$ and then a simple closed curve $\gamma$ on $\Sigma_{a_1}$ shrinks to a point $x$. 
We want to consider a simple closed curve $\tilde \gamma$ on $\Sigma$ that corresponds to $\gamma$ by the surgery from $\Sigma$ to $\Sigma_{a_1}$.
However if an isotopy of $\gamma$ on $\Sigma_{a_1}$ passes through the two disks of the surgery on $\Sigma_{a_1}$, then handle-slides along the vanishing cycle $a_1$ on $\Sigma$ are applied to $\tilde \gamma$.
Therefore we may only define the vanishing cycle $\tilde \gamma$ for $a_2$ on $\Sigma$ up to handle-slides along $a_1$.
For simplicity, we denote this vanishing cycle again by $a_2$.
We define the vanishing cycles $b_2$ and $c_2$ similarly.

Continuing the same observation, we may define the vanishing cycles $a_i, b_i, c_i$ on $\Sigma$ for $i = 3, \cdots, g - k$ up to isotopies and handle-slides over $a_1,\ldots,a_{i-1}$. 
These handle-slides are called upper-triangular handle-slides, see \cite[Remark 7.2]{BS}.

The advantage of considering the standard reference path is that we may study the global monodromies along  simple closed curves on the target space $\R^2$ up to upper-triangular handle-slides.
For example, Hayano proved the following lemma.
Let $\mu_i$ be the monodromy diffeomorphism from the fiber over $p_i$ to itself along a simple closed curve between $\Delta_{i}$ and $\Delta_{i+1}$, starting from $p_i$, parallel to these cusped curves and oriented counter-clockwise.

\begin{lem}[Hayano, \cite{hayano}]\label{monodromy}

Let $f:X\to \R^2$ be a simplified $(g, k)$-trisection and $i \in \{ 1, 2, \ldots, g - k \}$.
Suppose that $\mu_{i - 1}$ is trivial if $i > 1$.

\begin{enumerate}

\item 
Using isotopy of $a_i, b_i$ and $c_i$, we may assume that a regular neighborhood of $a_i \cup b_i \cup c_i$ is a genus--$1$ surface with three boundary components. 

\item 
The monodromy $\mu_{i}$ is $t_{\delta_1}^2t_{\delta_3}^2t_{\delta_2}^{-1}$ if we can take orientations of $a_{i}, b_{i}$ and $c_{i}$ so that the algebraic intersections $a_{i} \cdot b_{i}$, $b_{i} \cdot c_{i}$ and $c_{i} \cdot a_{i}$ are all equal to $1$ (see on the left in  Figure~\ref{genus1bd3}), and is equal to $t_{\delta_1}^{-2}t_{\delta_3}^{-2}t_{\delta_2}$ otherwise (see on the right), where $\delta_1, \delta_2$ and $\delta_3$ are the boundary components of the neighborhood of $a_{i}\cup b_{i}\cup c_{i}$ as in the figure and, for $j =1, 2, 3$, $t_{\delta_j}$ is the right-handed Dehn twist along $\delta_j$.

\end{enumerate}

\begin{figure}[htbp]
\begin{center}
\begin{minipage}{0.48\columnwidth}
\begin{center}
\includegraphics[clip, width=5.0cm, bb=217 606 376 711]{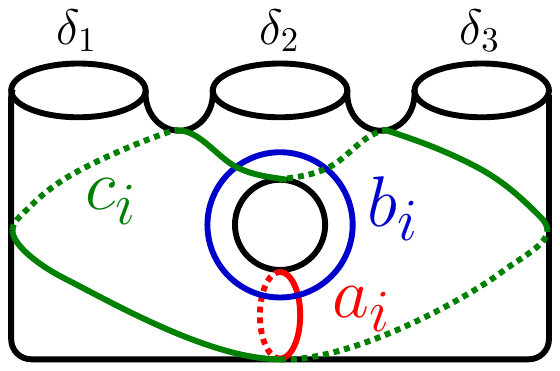}
\end{center}
\end{minipage}
\begin{minipage}{0.48\columnwidth}
\begin{center}
\includegraphics[clip, width=5.0cm, bb=217 606 376 711]{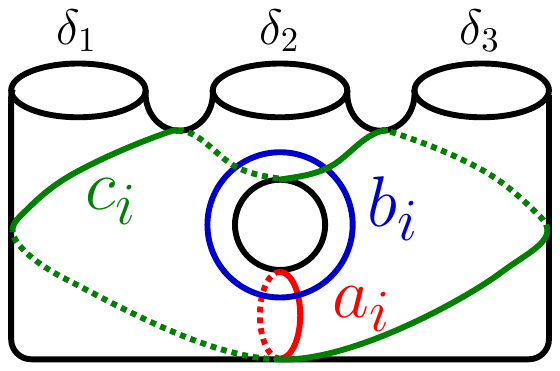}
\end{center}
\end{minipage}
\caption{Two possibilities of configurations of vanishing cycles of $f$. }
\label{genus1bd3}
\end{center}
\end{figure}
\end{lem}

 In this paper, we mainly study simplified $(2, 0)$-trisections. 
In this case, the fiber $\Sigma$ over the center $p_0$ is a genus-$2$ surface.
We may isotope $a_1, b_1$ and $c_1$ so that they intersect each other exactly once transversly as in Figure~\ref{genus1bd3}.
In this setting, $\Sigma$ is obtained from a regular neighborhood of $a_1 \cup b_1 \cup c_1$ by attaching either
\begin{itemize}
	\item two disks and one genus-$1$ surface with one boundary component or
	\item one disk and one annulus $A$.
\end{itemize}
In the first case, $\mu_1$ is the identity map of $\Sigma_{a_1}$ since all of $\delta_1, \delta_2$ and $\delta_3$ in Lemma~\ref{monodromy} are null-homotopic on $\Sigma_{a_1}$.
In the second case, let $d$ be a simple closed curve on $\Sigma_{a_1}$ obtained from the core of the annulus $A$ on $\Sigma$ by the surgery from $\Sigma$ to $\Sigma_{a_1}$.  
By Lemma \ref{monodromy}, the monodromy $\mu_1$ is either $t^{\pm 1}_d$ or $t^{\pm 4}_d$, where $t_d$ is the right-handed Dehn twist along $d$. 


Next, we introduce $6$-tuples of vertical $3$-manifolds and a classification of $6$-tuples shown in \cite{asano}.
Let $f : X \to \R^2$ be a simplified $(2, 0)$-trisection. 
Let $\gamma_{aa}$ be a properly embedded arc on $D^2 ( = f(X))$ that intersects $\Delta_1 \cup \Delta_2$ only at two points on the edge $e_a$ and separates $e_b \cup e_c$ and $\Delta_1$, where $e_a, e_b$ and $e_c$ are the edges on $\Delta_2$ in Figure~\ref{verTripleArc}.
The arcs $\gamma_{bb}, \gamma_{cc}$ are defined similarly. 
See the left figure in Figure~\ref{verTripleArc}.  
Let $\gamma_{ba}$ be a properly embedded arc on $D^2$ that intersects $\Delta_1 \cup \Delta_2$ twice,  at a point on $e_a$ and a point on $e_b$,  and separates $e_c$ and $\Delta_1$. 
The arcs $\gamma_{cb}, \gamma_{ac}$ are defined similarly. See the right figure in Figure~\ref{verTripleArc}. 
We set counter-clockwise orientations to these arcs. Set $V_{ij} = f^{-1}(\gamma_{ij})$ for $(i, j) \in \{ (a, a), (b, b), (c, c), (b, a), (c, b), (a, c) \}$ and set the orientations of $V_{ij}$ so that it coincides with the product of the orientations of the fiber and the arc $\gamma_{ij}$.
A $3$--manifold obtained as the preimage of an arc on $D^2$, as $V_{ij}$, is called a vertical $3$--manifold.
In \cite{asano} we focused on the $6$-tuples $\begin{pmatrix}
V_{aa} & V_{bb} & V_{cc}\\
V_{ba} & V_{cb} & V_{ac}
\end{pmatrix}$ of vertical $3$--manifolds to study simplified $(2, 0)$-trisections.

\begin{figure}[htbp]
\begin{center}
\includegraphics[clip, width=10cm, bb=129 533 517 703]{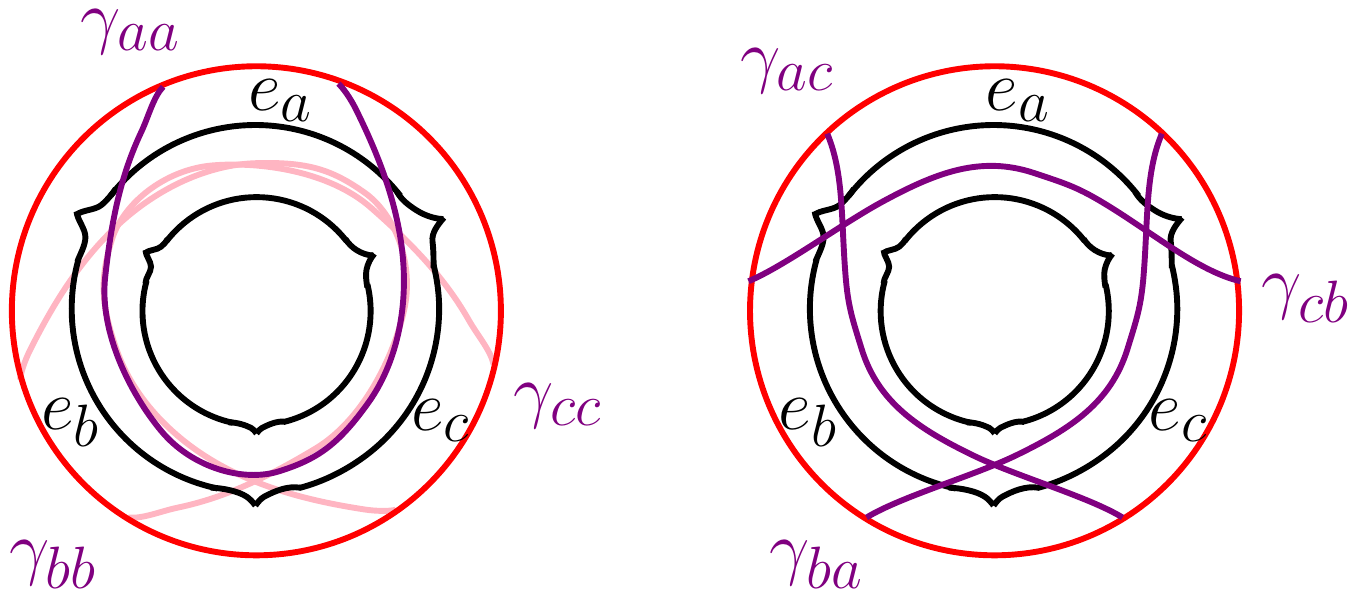}
\caption{The $6$-tuple of arcs.}
\label{verTripleArc}
\end{center}
\end{figure}

The orientation reversing diffeomorphism of $D^2$ sends 
the $6$-tuple $\begin{pmatrix}
V_{aa} & V_{bb} & V_{cc}\\
V_{ba} & V_{cb} & V_{ac}
\end{pmatrix}$ to  $\begin{pmatrix}
\bar{V}_{aa} & \bar{V}_{cc} & \bar{V}_{bb}\\
\bar{V}_{ac} & \bar{V}_{cb} & \bar{V}_{ba}
\end{pmatrix}$, where $\bar{V}_{ij}$ is the mirror image of $V_{ij}$. This operation corresponds to the exchange of the labels $b$ and $c$. 
We call it a reflection. 
Note that the reflection reverses the orientation of $X$ and those of the vertical $3$-manifolds in the $6$-tuple and, 
since we reverse the orientations of arcs so that they become counter-clockwise after the reflection, it does not change the orientations of the fibers.
The rotation of $D^2$ gives a cyclic permutation of the union of the arcs $\gamma_{ij}$ of order $3$, which maps $(a, b, c)$ to $(b, c, a)$.





The $6$-tuples are classified as follows:

\begin{thm}[\cite{asano}]\label{thmA}
The $6$-tuple $\begin{pmatrix}
V_{aa} & V_{bb} & V_{cc}\\
V_{ba} & V_{cb} & V_{ac}
\end{pmatrix}$ is one of the followings up to reflection and cyclic permutation:
\[
\begin{split}
&\begin{pmatrix}
S^1 \times S^2 & S^1 \times S^2 & S^1 \times S^2\\
S^3 & S^3 &S^3
\end{pmatrix}, \qquad 
\begin{pmatrix}
S^3 & S^3 & L((q - 1)^2, \epsilon q)\\
S^1 \times S^2 & L(q - 2, \epsilon) & L(q, -\epsilon)
\end{pmatrix}
, \\
&\begin{pmatrix}
S^3 & L(9, 2\epsilon) & L(4, \epsilon)\\
L(2, 1) & L(5, \epsilon) & S^3
\end{pmatrix}, \hspace{13.5mm}
\begin{pmatrix}
S^1 \times S^2 & L(4, 1) & L(4, 1)\\
S^3 & L(4 + \epsilon, 1) & S^3
\end{pmatrix}
, \\
&\begin{pmatrix}
S^1 \times S^2 & S^3 & S^3\\
S^3 & L(1 + \epsilon, 1) & S^3
\end{pmatrix},
\end{split}
\] where $q \neq 1$ and $\epsilon \in \{-1, 1\}$.
\end{thm}

As mentioned, the monodromy $\mu_1$ is either the identity map,  $t^{\pm 1}_d$ or $t^{\pm 4}_d$.
The first $6$-tuple occurs when $\mu_1$ is the identity map.
The fourth $6$-tuple occurs when $\mu_1 = t^{\pm 4}_d$ and the other $6$-tuples occur when $\mu_1 =t^{\pm 1}_d$.

In the proof of this theorem, configurations of vanishing cycles on $\Sigma_{a_1}$ are observed explicitly for each of $6$-tuples. 
This information will be used in the proof of Proposition \ref{containOrder3} below.

\section{Replacements of reference paths}
We first introduce a reference path for a simplified $(g, k)$-trisection $f : X \to \R^2$.
\begin{dfn}
A tree $\tau$ on $\R^2$ is called a {\it reference path} if there exists an orientation preserving diffeomorphism from $\R^2$ to itself such that
	\begin{itemize}
		\item it preserves the singular value set and
		\item the image of $\tau$ is the standard reference path.
	\end{itemize}
\end{dfn}
As we did for the standard reference path, we assign the labels $a_i, b_i$ and $c_i$ to the intersection points of $\tau$ and $\Delta_i$ for $i = 1, 2, \ldots, g - k$ as follows:
	\begin{itemize}
		\item The intersection point of $\Delta_1$ and the edge of $\tau$ connecting $p_0$ and $p_1$ is labeled with $a_1$.
		\item The other two intersection points of $\Delta_1$ and $\tau$ are labeled with $b_1$ and $c_1$ such that $a_1, b_1, c_1$ are aligned counter-clockwise along $\Delta_1$. 
		\item For each $i = 2, \ldots, g -k$, the three intersection points of $\Delta_i$ and $\tau$ are labeled with $a_i, b_i, c_i$ such that  $a_i, b_i, c_i$ are aligned counter-clockwise along $\Delta_i$ 
		and the edge of $\tau$ connecting $p_{i - 2}$ and $p_{i - 1}$ is between the edges of $\tau$ adjacent to the points labeled with $b_i$ and $c_i$.
	\end{itemize}
\begin{dfn}
Let $f$ be a simplified $(2, 0)$-trisection and $\tau$ be a reference path.
Let $a_1, b_1, c_1, a_2, b_2$ and $c_2$ be vanishing cycles on the central fiber $\Sigma$ determined by $\tau$.
The diagram on $\Sigma$ consisting of the labeled simple closed curves $a_1, b_1, c_1, a_2, b_2$ and $c_2$ is called a {\it simplified trisection diagram}.
We denote it by $(\Sigma ; \{a_1, a_2\}, \{b_1, b_2\},  \{c_1, c_2\})$.
\end{dfn}

Remark that a simplified trisection diagram of a simplified $(2, 0)$-trisection depends on the choice of a reference path even if we see it up to upper-triangular handle-slides.

Next, we introduce moves of reference paths and the maps induced by these moves.
Let $f$ be a simplified $(2, 0)$-trisection.
The middle points of the edges of $\Delta_1$ and $\Delta_2$ connecting the cusps are labeled with $v^1_1, v^2_1, v^3_1, v^1_2, v^2_2, v^3_2$ as in Figure \ref{graph}.
Recall that the labels $a_1, b_1, c_1, a_2, b_2$, and $c_2$ are givn as follows: Let $\tau$ be a reference path. The label $a_1$ is the intersection point of $\tau$ and $\Delta_1$.
If $a_1 = v^i_1$, then $b_1 = v^{i+1}_1,  c_1 = v^{i + 2}_1$, where $i = 1, 2, 3$ mod $3$. The union of edges of $\tau$ adjacent to $\Delta_2$ decomposes the disk bounded by $\Delta_2$ into three pieces. The end point of $\tau$ that does not intersect the closure of the piece containing $\Delta_1$ is labeled by $a_2$. If $a_2 = v^i_2$, then $b_2 = v^{i+1}_2,  c_2 = v^{i + 2}_2$, where $i = 1, 2, 3$ mod $3$.

\begin{figure}[htbp]
\begin{center}
\includegraphics[clip, width=5.5cm, bb=231 580 363 712]{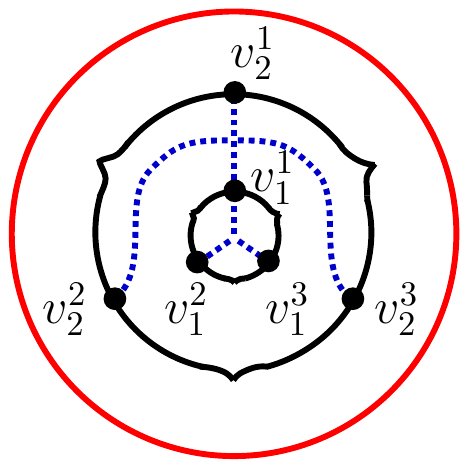}
\caption{The points $v^1_1, v^2_1, v^3_1, v^1_2, v^2_2$ and $v^3_2$.}\label{graph}
\end{center}
\end{figure}
Let $\mathcal{T}$ be the set of isotopy classes of reference paths, where the isotopy fixes the end points of reference paths.
\begin{dfn}
Let $\tau \in \mathcal{T}$, and let $\tau' \in \mathcal{T}$ be a reference path obtained from $\tau$ by replacing the arc $e_1$ connecting $p_1$ and $a_1=v^i_1$ by an arc $e_2$ connecting $p_1$ and $v^{i+1}_1$ 
such that $e_1 \cap e_2 = \{p_1\}$ and the region bounded by the union of $e_1, e_2$, the arc on $\tau$ connecting $p_0$ and $v_1^{i + 1}$ and that connecting $p_0$ and $v_1^{i}$ contains only one cusp of $\Delta_1$.
We call the map that sends $\tau$ to $\tau'$ a {\it $\Delta_1$-move}. 
See Figure~\ref{Delta1Move}.
\end{dfn}

\begin{figure}[htbp]
\begin{center}
\includegraphics[clip, width=11cm, bb=158 602 435 712]{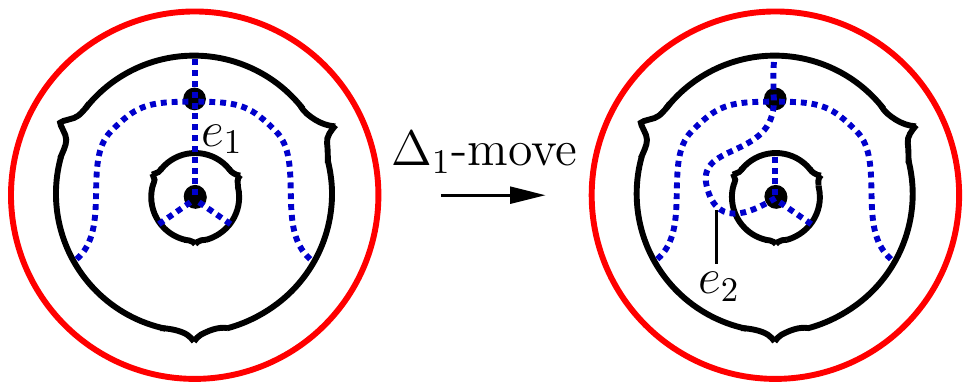}
\caption{$\Delta_1$-move.}\label{Delta1Move}
\end{center}
\end{figure}


\begin{dfn}\label{2}
Let $\tau \in \mathcal{T}$, and let $\tau' \in \mathcal{T}$ be a reference path obtained from $\tau$ by replacing the arc $e_1$ connecting $p_1$ and $c_2 = v^i_2$ by an arc $e_2$ connecting the same two points such that $e_1 \cap e_2 = \{p_1, v^i_2\}$ and the region bounded by the union of $e_1$ and $e_2$ contains $\Delta_1$. We call the map that sends $\tau$ to $\tau'$ a {\it $\Delta_2$-move}.
See Figure~\ref{Delta2Move}.
\end{dfn}

\begin{figure}[htbp]
\begin{center}
\includegraphics[clip, width=11cm, bb=158 587 435 697]{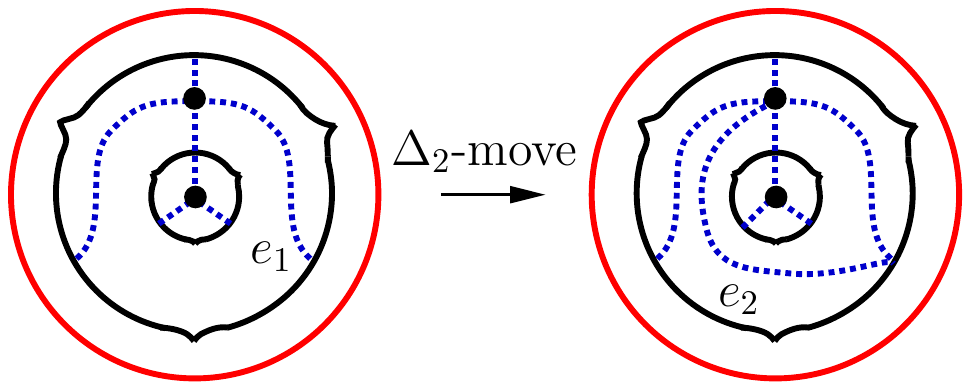}
\caption{$\Delta_2$-move.}\label{Delta2Move}
\end{center}
\end{figure}


\begin{lem}
Any $[\tau] \in \mathcal{T}$ is obtained from the standard referernce path by applying  $\Delta_1$-moves, $\Delta_2$-moves, and their inverses successively.
\end{lem}

\begin{proof}
The rotation in Figure~\ref{rotation} is obtained by applying $\Delta_1$-moves three times successively.
Therefore by using $\Delta_1$-moves or their inverses, we can assume that the path on $\tau$ connecting $a_1$ and $a_2$ is straight.
Then the standard reference path can be obtained by applying $\Delta_2$-moves or their inverses.
\end{proof}

\begin{figure}[htbp]
\begin{center}
\includegraphics[clip, width=13cm, bb=128 583 476 712]{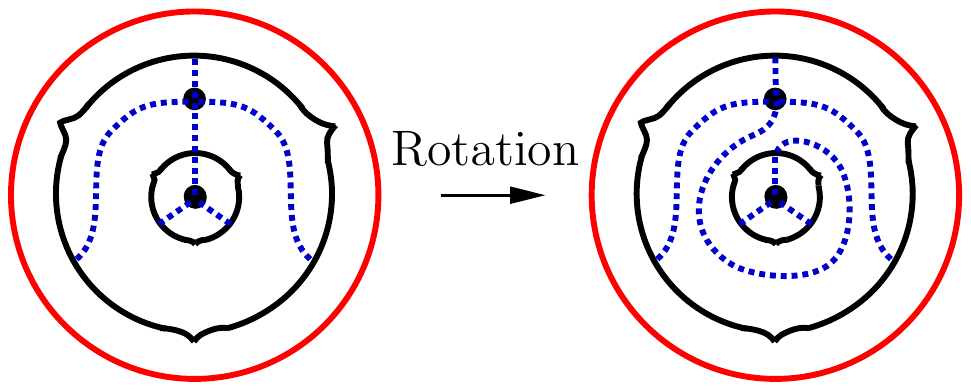}
\caption{Rotation.}
\label{rotation}
\end{center}
\end{figure}

Let $\mathcal{V}$ be the set of configurations of $6$ simple closed curves $\{ a_1, a_2 \}, \{ b_1, b_2 \}, \{ c_1, c_2 \}$ on the central fiber $\Sigma$ up to isotopy.
Let $\mathcal{V}_f$ be the subset of $\mathcal{V}$ consisting of simplified trisection diagrams of a simplified $(2, 0)$-trisection $f$.

Let $\sigma_1 : \mathcal{V} \to \mathcal{V}$ be the map induced by a $\Delta_1$-move and $\sigma_2 : \mathcal{V} \to \mathcal{V}$ be the map induced by a $\Delta_2$-move. 
Remark that the vanishing cycles $a_1, b_1, c_1$ on $\Sigma$ before applying a $\Delta_1$-move coincide with the vanishing cycles $b_1,  c_1, a_1,$ respectively, on $\Sigma$ after the $\Delta_1$-move.
Hence $\sigma_1(a_1) = b_1, \sigma_1(b_1) = c_1$ and $\sigma_1(c_1) = a_1$.
The vanishing cycles $a_2, b_2, c_2$ on $\Sigma$ after applying a $\Delta_1$-move are isotopic to the curves obtained from these vanishing cycles by applying the Dehn twist along the simple closed curve $t_{a_1}(b_1)$, where $t_a(b)$ means a simple closed curve obtained from $b$ by applying the right-handed Dehn twist along $a$, see \cite[Lemma 3.2]{hayano}.
Therefore we have $\sigma_1(a_2) = t_{t_{a_1}(b_1)}(a_2), \sigma_1(b_2) = t_{t_{a_1}(b_1)}(b_2)$ and $\sigma_1(c_2) = t_{t_{a_1}(b_1)}(c_2)$.

For a simple closed curve $w$ on $\Sigma$, we define $\mu_1^{-1}(w)$ as follows.
As mentioned, $\mu_1$ is either the identity map, $t_d^{\pm 1}$ or $t_d^{\pm 4}$.
If $\mu_1$ is the identity map, then we define $\mu_1^{-1}(w) = w$.
In the other cases, the core $d$ of the Dehn twist corresponds to one of $\delta_1, \delta_2$ and $\delta_3$ in Figure~\ref{genus1bd3}, which is a simple closed curve on $\Sigma$.
Suppose that it corresponds to $\delta_j$.
Then we define $\mu_1^{-1}(w) = t_{\delta_j}^{\mp 1}(w)$ if $\mu_1 = t_d^{\pm 1}$ and $\mu_1^{-1}(w) = t_{\delta_j}^{\mp 4}(w)$ if $\mu_1 = t_d^{\pm 4}$.

A $\Delta_2$-move does not change $a_1, b_1, c_1$.
The vanishing cycles $a_2, b_2$ and $c_2$ on $\Sigma$ after applying a $\Delta_2$-move are isotopic to $b_2, \mu_1^{-1}(c_2)$ and $a_2$, respectively, see \cite[Theorem 3.4]{hayano}.
Note that if $c_2$ is the vanishing cycle associated with the path from $p_0$ to $p_1$ and moving on $e_1$ in Definition~\ref{2} then $\mu_1^{-1}(c_2)$ is the vanishing cycle associated with the path from $p_0$ to $p_1$ and moving on $e_2$ in Definition~\ref{2}.
Thus we have $\sigma_2(a_2) = b_2, \sigma_2(b_2) =\mu_1^{-1}(c_2)$ and $\sigma_2(c_2) = a_2$.

In summary, we have 
\[ 
\begin{split}
\sigma_1(\Sigma; \{ a_1, a_2 \}, \{ b_1, b_2 \}, \{ c_1, c_2 \}) &= (\Sigma; \{ b_1, t_{t_{a_1}(b_1)}(a_2) \}, \{ c_1, t_{t_{a_1}(b_1)}(b_2) \}, \{ a_1, t_{t_{a_1}(b_1)}(c_2) \})\\
\sigma_2(\Sigma; \{ a_1, a_2 \}, \{ b_1, b_2 \}, \{ c_1, c_2 \}) &= (\Sigma; \{ a_1, b_2 \}, \{ b_1, \mu_1^{-1}(c_2) \}, \{ c_1, a_2 \}).
\end{split}
\]



We close this section with two lemmas.
\begin{lem}\label{trivial}
If the monodromy $\mu_1$ is the identity map, then for each $V$ in $\mathcal{V}_f$ and $i =1, 2$, there exists an automorphism of $\Sigma$ that sends $V$ to $\sigma_iV$.
\end{lem}
\begin{proof}
Since the monodromy $\mu_1$ is the identity map of $\Sigma_{a_1}$, the simplified trisection diagram $V$ is a connected sum of two tori $T_j$, $j = 1, 2$,  with diagrams consisting of three simple closed curves $a_j, b_j, c_j$ intersecting once each other transversely. 
On each torus $T_j$, we can confirm that there exists an automorphism of $T_j$ that sends the diagram to the diagram consisting of $\sigma_i(a_j), \sigma_i(b_j), \sigma_i(c_j)$.
Thus we have the assertion.
\end{proof}

\begin{lem}\label{differentRightLeftEq}
If two simplified $(2, 0)$-trisections are right-left equivalent, then the configurations of vanishing cycles are related by a finite sequence of automorphisms of $\Sigma$, upper-triangular handle-slides, and replacing reference paths.
\end{lem}
\begin{proof}
Let $f$ and $g$ be a pair of right-left equivalent simplified $(2, 0)$-trisections.
Hence there exist self-diffeomorphisms $\phi$ on $X$ and $\psi$ on $\R^2$ such that 
$$
\begin{CD}
X @> \phi >> X\\
@VfVV @VVgV \\
\R^2 @ > \psi >> \R^2
\end{CD}
$$ is commutative.
Let $\tau$ be a reference path for $f$ and we choose the reference path for $g$ as $\psi(\tau)$.
Let $(\Sigma; \{ a_1, a_2 \}, \{ b_1, b_2 \}, \{ c_1, c_2 \})$ and $(\tilde \Sigma; \{\tilde a_1, \tilde a_2 \}, \{\tilde b_1, \tilde b_2 \}, \{ \tilde c_1, \tilde c_2 \})$ be simplified trisection diagrams of $f$ and $g$, respectively.
Then the diffeomorphism $\phi$ restricted to the preimage of a disk on $\R^2$ containing $\Delta_1$ ensures that there exists a diffeomorphism from $\Sigma$ to $\tilde \Sigma$ that maps  $a_1, b_1$ and $c_1$ on $\Sigma$ to $\tilde a_1, \tilde b_1$ and $\tilde c_1$ on $\tilde \Sigma$, respectively, up to isotopy.
The surgery from $\Sigma$ to $\Sigma_{a_1}$ cuts $\Sigma$ along $a_1$ and fill the appearing boundary components by disks.
An isotopy of a vanishing cycle on $\Sigma_{a_1}$ that passes through these disks corresponds to handle-slides and this operation does not change the right-left equivalence class of the map.
The diffeomorphism $\phi$ ensures that there exists a diffeomorphism from $\Sigma_{a_1}$ to $\tilde \Sigma_{\tilde a_1}$ that maps  $a_2', b'_2$ and $c'_2$ on $\Sigma_{a_1}$ to $\tilde a'_2, \tilde b'_2$ and $\tilde c'_2$ on $\tilde \Sigma_{\tilde a_1}$, respectively, up to isotopy, where $a'_2, b'_2, c'_2$ are vanishing cycles on $\Sigma_{a_1}$ obtained from $a_2, b_2, c_2$ on $\Sigma$ by the surgery along $a_1$ and the notations $\tilde a_2', \tilde b_2', \tilde c_2'$ are those on $\tilde \Sigma_{\tilde a_1}$ obtained from $\tilde a_2, \tilde b_2, \tilde c_2$ similarly, respectively.
Hence there exists a diffeomorphism from $\Sigma$ to $\tilde \Sigma$ that maps  $a_2, b_2$ and $c_2$ on $\Sigma$ to $\tilde a_2, \tilde b_2$ and $\tilde c_2$ on $\tilde \Sigma$, respectively, up to isotopy and upper-triangular handle-slides.
This completes the proof.
\end{proof}

\section{Proof of Theorem 2}
Let $\eta$ and $\eta'$ be two maps from $\mathcal{V}_f$ to itself obtained by applying $\sigma_1, \sigma_2$ and their inverses successively.
We set $\eta \sim \eta'$ if for any $V \in \mathcal{V}_f$ there exists $h \in \mathrm{Mod(\Sigma)}$ such that $\eta(V) = h( \eta'(V))$, where $\Mod(\Sigma)$ is the mapping class group of $\Sigma$.
We define the group $G_f$ associated with $f$ by $\langle \sigma_1, \sigma_2 \mid \: \sim \rangle$.
Remark that, by Lemma~\ref{trivial}, if the monodromy $\mu_1$ is the identity map, then $G_f = \{ 1 \}$. 

Theorem 1 will follow from the next proposition.
\begin{prop}\label{containOrder3}
If the monodromy $\mu_1$ is not the identity map, then $G_f$ contains $\Z / 3\Z$ as a subgroup.
\end{prop}

Let $V =  (\Sigma; \{ a_1, a_2 \}, \{ b_1, b_2 \}, \{ c_1, c_2 \})$ be a simplified trisection diagram associated with a reference path. 
\begin{lem}\label{order3}
For any $V \in \mathcal{V}_f$, there exists an automorphism of $\Sigma$ that sends $V$ to $\sigma_2^3 V$. 
\end{lem}
\begin{proof}
	\[
	\begin{split}
		\sigma_2^3(\Sigma; \{ a_1, a_2 \}, \{ b_1, b_2 \}, \{ c_1, c_2 \}) &= \sigma_2^2(\Sigma; \{ a_1, b_2 \}, \{ b_1, \mu_1^{-1}(c_2) \}, \{ c_1, a_2 \} ).\\
		&= \sigma_2(\Sigma; \{ a_1, \mu_1^{-1}(c_2) \}, \{ b_1, \mu_1^{-1}(a_2) \}, \{ c_1, b_2 \} ).\\
		&= (\Sigma; \{ a_1, \mu_1^{-1}(a_2) \}, \{ b_1, \mu_1^{-1}(b_2) \}, \{ c_1, \mu_1^{-1}(c_2) \} ).
	\end{split}
	\]
As mentioned in Section 2, the monodromy $\mu_1$ is either the identity map, $t^{\pm 1}_d$ or $t^{\pm 4}_d$, and the core $d$ of the Dehn twist does not intersect $a_1, b_1$ and $c_1$.
Thus the assertion follows.
\end{proof}
Therefore, to prove Proposition~\ref{containOrder3}, it is enough to show that if $\mu_1$ is not the identity map, then $V, \sigma_2 V$ and $\sigma_2^2V$  are different modulo $\Mod(\Sigma)$ and upper-triangular handle-slides.
To do this, it is enough to show that if $\mu_1$ is not the identity map, then the configurations of vanishing cycles of $V, \sigma_2 V$ and $\sigma_2^2V$ are different modulo $\Mod(\Sigma_{a_1})$ after the surgery from $\Sigma$ to $\Sigma_{a_1}$.

We assign orientations to  $a_1, b_1$ and $c_1$ so that $a_1 \cdot b_1 = b_1 \cdot c_1 = c_1 \cdot a_1 = 1$ if $(a_1 \cdot b_1) (b_1 \cdot c_1) (c_1 \cdot a_1) = 1$, which is the case on the left in Figure~\ref{genus1bd3}, and $a_1 \cdot b_1 = b_1 \cdot c_1 = c_1 \cdot a_1 = -1$ if $(a_1 \cdot b_1) (b_1 \cdot c_1) (c_1 \cdot a_1) = -1$, which is the case on the right.
We assign orientations to $a_2, b_2$ and $c_2$ arbitrary.
Let $b_1'$ and $c_1'$ be the arcs on $\Sigma_{a_1}$ obtained from $b_1$ and $c_1$, respectively,  by the surgery from $\Sigma$ to $\Sigma_{a_1}$, where $a_1$ becomes two disks on $\Sigma_{a_1}$ by the surgery and $b_1'$ and $c_1'$ are arcs connecting these two disks.
Shrinking these two disks to points, we obtain an immersed curve on $\Sigma_{a_1}$ consisting of $b_1'$ and $c_1'$, which we denote by $b_1' + c_1'$. 
\begin{lem}\label{d}
	\begin{itemize}
		\item If $\Sigma$ is obtained from a regular neighborhood of $a_1 \cup b_1 \cup c_1$ by attaching two disks and one genus-$1$ surface with one boundary component, then $[b_1' + c_1'] = [d] = 0$ in $H_1(\Sigma_{a_1})$.
		\item If $\Sigma$ is obtained from a regular neighborhood of $a_1 \cup b_1 \cup c_1$ by attaching one disk and one annulus, then a simple closed curve representing the cycle $[b_1' + c_1'] \in H_1(\Sigma_{a_1})$ is parallel to $d$.
	\end{itemize}
\end{lem}
\begin{proof}
	We may confirm the assertions in Figure~\ref{genus1bd3}.
	Note that $d$ is parallel to one of $\delta_1, \delta_2$ and $\delta_3$ in the figure.
\end{proof}

We set
$$I(V) = \begin{pmatrix}
|(b_1 + c_1) \cdot a_2|\\
|(b_1 + c_1) \cdot b_2|\\
|(b_1 + c_1) \cdot c_2|\\
\end{pmatrix},$$
where $b_1 + c_1$ is the union of $b_1$ and $c_1$ on $\Sigma$.
For $w \in \{ a_2, b_2, c_2\}$, let $w'$ be the vanishing cycle on $\Sigma_{a_1}$ obtained from $w$ by the surgery from $\Sigma$ to $\Sigma_{a_1}$.
\begin{lem}\label{invariantForHS}
	If $V'$ is obtained from $V$ by upper-triangular handle-slides, then $I(V) = I(V')$. 
\end{lem}
\begin{proof}
	Since $(b_1 + c_1) \cdot w = (b'_1 + c'_1) \cdot w'$ for $w = a_2, b_2, c_2$, we have $|(b_1 + c_1) \cdot w| = |d \cdot w'|$ by 			Lemma~\ref{d}.
	Since $|d \cdot w'|$ is invariant under upper-triangular handle-slides of $w$, the assertion follows.
\end{proof}
\begin{lem}
$$I(\sigma_2^{n} V) = \begin{pmatrix}
						|(b_1 + c_1) \cdot \sigma^{n} a_2|\\
						|(b_1 + c_1) \cdot \sigma^{n} b_2|\\
						|(b_1 + c_1) \cdot \sigma^{n} c_2|\\
						\end{pmatrix},$$ where $\sigma$ is the cyclic permutation that sends $(a_2, b_2, c_2)$ to $(b_2, c_2, a_2)$. 
\end{lem}

\begin{proof}
		As mentioned in Section $3$,
						$$\sigma_2(\Sigma; \{ a_1, a_2 \}, \{ b_1, b_2 \}, \{ c_1, c_2 \}) = (\Sigma; \{ a_1, b_2 \}, \{ b_1, \mu_1^{-1}(c_2) \}, \{ c_1, a_2 \}).$$ When $n = 1$ we have
		$$I(\sigma_2 V) = \begin{pmatrix}
						|(b_1 + c_1) \cdot b_2|\\
						|(b_1 + c_1) \cdot \mu_1^{-1}(c_2)|\\
						|(b_1 + c_1) \cdot a_2|\\
						\end{pmatrix}
						= \begin{pmatrix}
						|(b_1 + c_1) \cdot b_2|\\
						|(b_1 + c_1) \cdot c_2|\\
						|(b_1 + c_1) \cdot a_2|\\						
						\end{pmatrix}
						= \begin{pmatrix}
						|(b_1 + c_1) \cdot \sigma a_2|\\
						|(b_1 + c_1) \cdot \sigma b_2|\\
						|(b_1 + c_1) \cdot \sigma c_2|\\
						\end{pmatrix},$$where we used the fact that $\delta_1, \delta_2$ and $\delta_3$ in Figure~\ref{genus1bd3}, some of which corresponds to the core of the Dehn twist of $\mu_1$ on $\Sigma_{a_1}$, do not intersect $b_1$ and $c_1$.
						Remark that we do not need to care the orientations of $a_2, b_2$ and $c_2$ because $\sigma_2$ does not change $a_1, b_1, c_1$ and the entries in the definition of $I(V)$ are absolute values.
						For induction, we assume that the following equality holds in the case $n = k$:
						\begin{align*}
						I(\sigma_2^{k} V) =
						\left\{
						\begin{array}{l}
						\vspace{3mm}
						\begin{pmatrix}
						|(b_1 + c_1) \cdot a_2|\\
						|(b_1 + c_1) \cdot b_2|\\
						|(b_1 + c_1) \cdot c_2|\\
						\end{pmatrix} (k \equiv 0 \ \mathrm{mod} \ 3)\\
						\vspace{3mm}			
						\begin{pmatrix}
						|(b_1 + c_1) \cdot b_2|\\
						|(b_1 + c_1) \cdot c_2|\\
						|(b_1 + c_1) \cdot a_2|\\
						\end{pmatrix}(k \equiv 1 \ \mathrm{mod} \ 3)\\
						\begin{pmatrix}
						|(b_1 + c_1) \cdot c_2|\\
						|(b_1 + c_1) \cdot a_2|\\
						|(b_1 + c_1) \cdot b_2|\\
						\end{pmatrix}(k \equiv 2 \ \mathrm{mod} \ 3).
						\end{array}
						\right.
						\end{align*}
						The image of $\sigma_2$ is obtained by exchanging the second entries of three brackets in $(\Sigma; \{ a_1, a_2 \}, \{ b_1, b_2 \}, \{ c_1, c_2 \})$ by the cyclic permutation $\sigma$ and then applying $\mu_1^{-1}$ to the second entry in the second bracket.
						Here we again used the fact that $\delta_1, \delta_2$ and $\delta_3$ do not intersect $b_1$ and $c_1$.
						Thus we have
						\begin{align*}
						I(\sigma_2^{k + 1} V) =
						\left\{
						\begin{array}{l}
						\vspace{3mm}
						\begin{pmatrix}
						|(b_1 + c_1) \cdot b_2|\\
						|(b_1 + c_1) \cdot c_2|\\
						|(b_1 + c_1) \cdot a_2|\\
						\end{pmatrix} (k \equiv 0 \ \mathrm{mod} \ 3)\\
						\vspace{3mm}			
						\begin{pmatrix}
						|(b_1 + c_1) \cdot c_2|\\
						|(b_1 + c_1) \cdot a_2|\\
						|(b_1 + c_1) \cdot b_2|\\
						\end{pmatrix}(k \equiv 1 \ \mathrm{mod} \ 3)\\
						\begin{pmatrix}
						|(b_1 + c_1) \cdot a_2|\\
						|(b_1 + c_1) \cdot b_2|\\
						|(b_1 + c_1) \cdot c_2|\\
						\end{pmatrix}(k \equiv 2 \ \mathrm{mod} \ 3)
						\end{array}
						\right.
						= \begin{pmatrix}
						|(b_1 + c_1) \cdot \sigma^{k + 1} a_2|\\
						|(b_1 + c_1) \cdot \sigma^{k + 1} b_2|\\
						|(b_1 + c_1) \cdot \sigma^{k + 1} c_2|\\
						\end{pmatrix}.						
						\end{align*}
						This completes the proof.
\end{proof}

\begin{lem}\label{except}
	Let $V$ be a simplified trisection diagram of a simplified $(2, 0)$-trisection. 
	\begin{itemize}
		\item[(1)] $I(\sigma_1V) = I(V)$.
		\item[(2)] $I(\sigma_2V) \neq I(V)$ unless either $b'_2$ and $c'_2$ are parallel or $a'_2$ and $\mu_1^{-1}(c'_2)$ are parallel.
	\end{itemize}
\end{lem}

\begin{proof}
First we prove (1). As the mentioned in Section 3, $\sigma_1$ satisfies
\[
\sigma_1(\Sigma; \{ a_1, a_2 \}, \{ b_1, b_2 \}, \{ c_1, c_2 \}) = (\Sigma; \{ b_1, t_{t_{a_1}(b_1)}(a_2) \}, \{ c_1, t_{t_{a_1}(b_1)}(b_2) \}, \{ a_1, t_{t_{a_1}(b_1)}(c_2) \}).
\]
If $a_1 \cdot b_1 = b_1 \cdot c_1 = c_1 \cdot a_1 = 1$, 
then we have 
\[
\begin{split}
\sigma_1b_1 &= t_{t_{b_1}(c_1)}(b_1) = b_1 - (b_1 \cdot (c_1 - (c_1 \cdot b_1)b_1))(c_1 - (c_1 \cdot b_1)b_1) = -c_1\\
\sigma_1c_1 &= t_{t_{c_1}(a_1)}(c_1) = c_1 - (c_1 \cdot (a_1 - (a_1 \cdot c_1)c_1))(a_1 - (a_1 \cdot c_1)c_1) = -a_1.
\end{split}
\]
We also have
\[
\begin{split}
c_1 \cdot t_{t_{a_1}(b_1)}(a_2) &= c_1 \cdot (a_2 - (a_2 \cdot (b_1 - (b_1 \cdot a_1)a_1))(b_1 - (b_1 \cdot a_1)a_1))\\
							&= c_1\cdot  a_2 - (a_2 \cdot b_1) (c_1\cdot  b_1) + (b_1\cdot  a_1)(a_2 \cdot b_1)(c_1 \cdot a_1)\\
							&= c_1 \cdot a_2
\end{split}
\] 
and
\[
\begin{split}
a_1 \cdot t_{t_{a_1}(b_1)}(a_2) &= a_1 \cdot (a_2 - (a_2 \cdot (b_1 - (b_1 \cdot a_1)a_1))(b_1 - (b_1 \cdot a_1)a_1))\\
							&= a_1\cdot  a_2 - (a_2 \cdot b_1) (a_1\cdot  b_1) + (b_1\cdot  a_1)(a_2 \cdot b_1)(a_1 \cdot a_1)\\
							&= b_1 \cdot a_2
\end{split}
\] 
and similarly
\[
\begin{split}
c_1 \cdot t_{t_{a_1}(b_1)}(b_2) &= c_1 \cdot b_2\\
a_1 \cdot t_{t_{a_1}(b_1)}(b_2) &=  b_1 \cdot b_2\\
c_1 \cdot t_{t_{a_1}(b_1)}(c_2) &= c_1 \cdot c_2\\
a_1 \cdot t_{t_{a_1}(b_1)}(c_2) &= b_1 \cdot c_2.
\end{split}
\]
Here we used the fact that $a_1$ does not intersect $a_2, b_2$ and $c_2$.
Thus we have
		$$I(\sigma_1V) = 
		\begin{pmatrix}
		|(c_1 + a_1) \cdot t_{t_{a_1}(b_1)}(a_2)|\\
		|(c_1 + a_1) \cdot t_{t_{a_1}(b_1)}(b_2)|\\
		|(c_1 + a_1) \cdot t_{t_{a_1}(b_1)}(c_2)|\\						
		\end{pmatrix} =
		\begin{pmatrix}
		|(c_1 + b_1) \cdot a_2|\\
		|(c_1 + b_1) \cdot b_2|\\
		|(c_1 + b_1) \cdot c_2|\\						
		\end{pmatrix} = 
		I(V).$$
If $a_1 \cdot b_1 = b_1 \cdot c_1 = c_1 \cdot a_1 = -1$, by similar caluculation, 
we have
$\sigma_1 b_1 = c_1, \sigma_1 c_1 = a_1$ and  
\[
\begin{split}
c_1 \cdot t_{t_{a_1}(b_1)}(a_2) &= c_1 \cdot a_2 + 2 (b_1 \cdot a_2)\\
a_1 \cdot t_{t_{a_1}(b_1)}(a_2) &= - b_1 \cdot a_2\\
c_1 \cdot t_{t_{a_1}(b_1)}(b_2) &= c_1 \cdot b_2 + 2 (b_1 \cdot b_2)\\
a_1 \cdot t_{t_{a_1}(b_1)}(b_2) &=  - b_1 \cdot b_2\\
c_1 \cdot t_{t_{a_1}(b_1)}(c_2) &= c_1 \cdot b_2 + 2 (b_1 \cdot c_2)\\
a_1 \cdot t_{t_{a_1}(b_1)}(c_2) &= - b_1 \cdot c_2.
\end{split}
\]
Thus $I(\sigma_1V) =  I(V)$ again holds.

Next we prove (2).
		In the following proofs, the notation of double sign correspondence is used.
		Suppose that the $6$-tuple is 
		$\begin{pmatrix}
		S^3 & S^3 & L((q - 1)^2, \epsilon q)\\
		S^1 \times S^2 & L(q - 2, \epsilon) & L(q, -\epsilon)
		\end{pmatrix}$ in the five cases in Theorem~\ref{thmA}.
		In this case, $\mu_1 = t_d^{\pm1}$.
		From \cite{asano}, we can assume that the configurations of vanishing cycles on $\Sigma_{a_1}$ are
		\[
		[a'_2] = \begin{pmatrix} 1 \\ 0 \end{pmatrix},\quad
		[b'_2] = \begin{pmatrix} 0 \\ 1 \end{pmatrix},\quad
		[c'_2] = \begin{pmatrix} \pm q - 2 \\ \pm 1 \end{pmatrix},\quad
		[d] = \begin{pmatrix} \mp 1 \\  1 \end{pmatrix}, 
		\]where $p, q \in \mathbf{Z}$. 
		Then, by Lemma~\ref{d}, we have
		$$I(V) = 
		\begin{pmatrix}
		|(b_1 + c_1) \cdot a_2|\\
		|(b_1 + c_1) \cdot b_2|\\
		|(b_1 + c_1) \cdot c_2|\\
		\end{pmatrix} = 
		\begin{pmatrix}
		1\\
		1\\
		|1 \mp q|\\
		\end{pmatrix},\; 
		I(\sigma_2V) = 
		\begin{pmatrix}
		|(b_1 + c_1) \cdot b_2|\\
		|(b_1 + c_1) \cdot c_2|\\
		|(b_1 + c_1) \cdot a_2|\\
		\end{pmatrix} = 
		\begin{pmatrix}
		1\\
		|1 \mp q|\\
		1\\
		\end{pmatrix}.$$
		Thus we obtain $I(\sigma_2V) \neq  I(V)$ except for the case $|1 \mp q| = 1$.
		These exceptional cases are excluded in the assertion.
		
		Suppose that the $6$-tuple is 
		$\begin{pmatrix}
		S^3 & L(9, 2\epsilon) & L(4, \epsilon)\\
		L(2, 1) & L(5, \epsilon) & S^3
		\end{pmatrix}$.
		In this case, $\mu_1 = t_d^{\pm1}$.
		From \cite{asano}, we can assume that the configurations of vanishing cycles are
		\[
		[a'_2] = \begin{pmatrix} 1 \\ 0 \end{pmatrix},\quad
		[b'_2] = \begin{pmatrix} 0 \\ 1 \end{pmatrix},\quad
		[c'_2] = \begin{pmatrix} 5 \\ \mp1 \end{pmatrix},\quad
		[d] = \begin{pmatrix} r \\  1 \end{pmatrix}, 
		\]where $r = \mp 2$ or $\mp 3$.
		Then we have
		$$I(V) = 
		\begin{pmatrix}
		|(b_1 + c_1) \cdot a_2|\\
		|(b_1 + c_1) \cdot b_2|\\
		|(b_1 + c_1) \cdot c_2|\\
		\end{pmatrix} = 
		\begin{pmatrix}
		1\\
		|r| \\
		5 - |r|\\
		\end{pmatrix}, \;
		I(\sigma_2V) = 
		\begin{pmatrix}
		|(b_1 + c_1) \cdot b_2|\\
		|(b_1 + c_1) \cdot c_2|\\
		|(b_1 + c_1) \cdot a_2|\\
		\end{pmatrix} = 
		\begin{pmatrix}
		|r|\\
		5 - |r|\\
		1\\
		\end{pmatrix}.$$
		Thus we obtain $I(\sigma_2V) \neq I(V)$.
			
		Suppose that the $6$-tuple is 
		$\begin{pmatrix}
		S^1 \times S^2 & L(4, 1) & L(4, 1)\\
		S^3 & L(4 + \epsilon, 1) & S^3
		\end{pmatrix}$.
		In this case, $\mu_1 = t_d^{\pm4}$.		
		From \cite{asano}, we can assume that the configurations of vanishing cycles are
		\[
		[a'_2] = \begin{pmatrix} 1 \\ 0 \end{pmatrix},\quad
		[b'_2] = \begin{pmatrix} 0 \\ 1 \end{pmatrix},\quad
		[c'_2] = \begin{pmatrix} -1 \pm 4 \epsilon_2 \\ \epsilon_2 \end{pmatrix},\quad
		[d] = \begin{pmatrix} 1 \\  0 \end{pmatrix}, 
		\]where $\epsilon_2 \in \{ 1, -1 \}$. 
		Then we have
		$$I(V) = 
		\begin{pmatrix}
		|(b_1 + c_1) \cdot a_2|\\
		|(b_1 + c_1) \cdot b_2|\\
		|(b_1 + c_1) \cdot c_2|\\
		\end{pmatrix} = 
		\begin{pmatrix}
		0\\
		1\\
		1\\
		\end{pmatrix}, \;
		I(\sigma_2V) = 
		\begin{pmatrix}
		|(b_1 + c_1) \cdot b_2|\\
		|(b_1 + c_1) \cdot c_2|\\
		|(b_1 + c_1) \cdot a_2|\\
		\end{pmatrix} = 
		\begin{pmatrix}
		1\\
		1\\
		0\\
		\end{pmatrix}.$$
		Thus we obtain $I(\sigma_2V) \neq I(V)$.
		
		Suppose that the $6$-tuple is
		$\begin{pmatrix}
		S^1 \times S^2 & S^3 & S^3\\
		S^3 & S^1 \times S^2 & S^3
		\end{pmatrix}$.
		In this case, $\mu_1 = t_d^{\pm1}$.
		From \cite{asano}, we can assume that the configurations of vanishing cycles are
		\[
		[a'_2] = \begin{pmatrix} 1 \\ 0 \end{pmatrix},\quad
		[b'_2] = \begin{pmatrix} 0 \\ 1 \end{pmatrix},\quad
		[c'_2] = \begin{pmatrix} 0 \\ \pm 1 \end{pmatrix},\quad
		[d] = \begin{pmatrix} 1 \\  0 \end{pmatrix}, 
		\]
		Since $b'_2$ and $c'_2$ are parallel, this case is excluded.
		Note that, even though this case is excluded, we have $I(\sigma_2V) \neq I(V)$ since
		$$I(\Sigma) = 
		\begin{pmatrix}
		|(b_1 + c_1) \cdot a_2|\\
		|(b_1 + c_1) \cdot b_2|\\
		|(b_1 + c_1) \cdot c_2|\\
		\end{pmatrix} = 
		\begin{pmatrix}
		0\\
		1\\
		1\\
		\end{pmatrix}, \;\;
		I(\sigma_2V) = 
		\begin{pmatrix}
		|(b_1 + c_1) \cdot b_2|\\
		|(b_1 + c_1) \cdot c_2|\\
		|(b_1 + c_1) \cdot a_2|\\
		\end{pmatrix} = 
		\begin{pmatrix}
		1\\
		1\\
		0\\
		\end{pmatrix}.$$
		
		Suppose that the $6$-tuple is
		$\begin{pmatrix}
		S^1 \times S^2 & S^3 & S^3\\
		S^3 & L(2, 1) & S^3
		\end{pmatrix}$.
		In this case, $\mu_1 = t_d^{\pm1}$.
		From \cite{asano}, we can assume that the configurations of vanishing cycles are
		\[
		[a'_2] = \begin{pmatrix} 1 \\ 0 \end{pmatrix},\quad
		[b'_2] = \begin{pmatrix} 0 \\ 1 \end{pmatrix},\quad
		[c'_2] = \begin{pmatrix} -2 \\ \mp 1 \end{pmatrix},\quad
		[d] = \begin{pmatrix} 1 \\  0 \end{pmatrix}. 
		\] 
		Then we have
		$$I(V) = 
		\begin{pmatrix}
		|(b_1 + c_1) \cdot a_2|\\
		|(b_1 + c_1) \cdot b_2|\\
		|(b_1 + c_1) \cdot c_2|\\
		\end{pmatrix} = 
		\begin{pmatrix}
		0\\
		1\\
		1\\
		\end{pmatrix}, \;
		I(\sigma_2V) = 
		\begin{pmatrix}
		|(b_1 + c_1) \cdot b_2|\\
		|(b_1 + c_1) \cdot c_2|\\
		|(b_1 + c_1) \cdot a_2|\\
		\end{pmatrix} = 
		\begin{pmatrix}
		1\\
		1\\
		0\\
		\end{pmatrix}.$$ Thus we obtain $I(\sigma_2V) \neq I(V)$.
		This completes the proof.
\end{proof}


\begin{proof}[\rm{{\bf Proof of Proposition \ref{containOrder3}.}}]
If $f$ is a not in the exceptional cases in Lemma~\ref{except}, then it satisfies
\[
I(V) \neq I(\sigma_2 V) \neq I(\sigma_2^2 V).
\]
Hence, with Lemma~\ref{order3}, we obtain $\langle \sigma_2 \rangle \simeq \Z / 3\Z \subset G_f$.
\end{proof}

\begin{proof}[\rm{{\bf Proof of Thorem 1}}]
Let $f : X \to \R^2$ be a simplified $(2, 0)$-trisection and $\Sigma$ be its simplified trisection diagram.
Let $f'$ and $f''$ be simplified $(2, 0)$-trisections obtained from $f$ by applying one $\Delta_1$-move and two $\Delta_1$-moves, respectively. 
In particular, $f, f'$ and $f''$ are right-left equivalent.
However, as mentioned in the proof of Proposition~\ref{containOrder3}, $I(V) \neq I(\sigma_2 V) \neq I(\sigma_2^2 V)$ holds unless $f$ is in the exceptional cases in the assertion.
Hence, by Lemma~\ref{invariantForHS}, $V, \sigma_2 V$ and $\sigma^2_2 V$ are not related by automorphisms of $\Sigma$ and upper-triangular handle-slides.
\end{proof}


\end{document}